\documentclass[a4paper,pdftex,reqno,10pt]{amsart}

\allowdisplaybreaks[1]

\usepackage{geometry}
\usepackage{graphicx}
\usepackage{enumerate}
\usepackage{courier}
\usepackage{times}
\usepackage{amssymb}
\usepackage{amsmath}

\usepackage{hyperref}

\theoremstyle{plain}
\newtheorem{Theorem}{Theorem}[section]

\newtheorem{Corollary}[Theorem]{Corollary}
\newtheorem{Lemma}[Theorem]{Lemma}

\newtheorem{Observation}[Theorem]{Observation}

\newenvironment{Proof}
{\begin{trivlist}\item[]{{\sc Proof.}}}{\hfill{$\square$}\noindent\end{trivlist}}

\newcommand{\gaussm}[3]{\genfrac{[}{]}{0pt}{}{#1}{#2}_{#3}}
\newcommand{\PG}[2]{\operatorname{PG}(#1,#2)}
\newcommand{\F}[2]{\mathbb{F}_{#2}^{#1}}

\theoremstyle{definition}

\theoremstyle{remark}

\begin{document}

%%%%%%%%%%%%%%%%%%%%%%%%%%%%%%%%%%%%%%%%%%%%%%%%%%%%%%%%%%%%%%%%%%%%%%%%%%%%%%
%% Title:
%%%%%%%%%%%%%%%%%%%%%%%%%%%%%%%%%%%%%%%%%%%%%%%%%%%%%%%%%%%%%%%%%%%%%%%%%%%%%%

\title{Improved upper bounds for partial spreads}

\author{Sascha Kurz$^\star$}
\address{Department of Mathematics, University of Bayreuth, 95440 Bayreuth, Germany}
\email{sascha.kurz@uni-bayreuth.de}
\thanks{$^\star$ The work of the author was supported by the ICT COST Action IC1104
and grant KU 2430/3-1 -- Integer Linear Programming Models for Subspace Codes and Finite Geometry
from the German Research Foundation.}

\date{}

\begin{abstract}
  A \emph{partial $(k-1)$-spread} in $\PG{n-1}{q}$ is a collection of $(k-1)$-dimensional 
  subspaces with trivial intersection, i.e., each \emph{point} is covered at most once.
  So far the maximum size of a partial $(k-1)$-spread in $\PG{n-1}{q}$ was known 
  for the cases $n\equiv 0\pmod k$, $n\equiv 1\pmod k$ and $n\equiv 2\pmod k$ with the 
  additional requirements $q=2$ and $k=3$. We completely resolve the case 
  $n\equiv 2\pmod k$ for the binary case $q=2$.
  %% We determine the maximum size of a partial $k$-spread in $\F{n}{2}$ for $k\ge 4$ 
  %%and $n\equiv 2\pmod k$. 
  
  \medskip
  
  %% 51E23 Spreads and packing problems
  %% 94B25 Combinatorial codes
  %% 05B15 Orthogonal arrays, Latin squares, Room squares
  %% 05B40 Packing and covering
  %% 11T71 Algebraic coding theory; cryptography 
\end{abstract}

\maketitle

\section{Introduction}

\noindent
For a prime power $q>1$ let $\mathbb{F}_q$ be the finite field with
$q$ elements and $\mathbb{F}_q^n$ the standard vector space of dimension
$n\ge 1$ over $\mathbb{F}_q$. The set of all subspaces of $\mathbb{F}_q^n$,  
ordered by the incidence relation $\subseteq$, is called \emph{$(n-1)$-dimensional 
projective geometry over $\mathbb{F}_q$} and commonly denoted by $\PG{n-1}{q}$.
Let $G_q(n,k)$ denote the set of all $k$-dimensional subspaces in $\F{n}{q}$.\footnote{Instead 
of $\PG{n-1}{q}$ we will mainly use the notation $\F{n}{q}$ in the following.}
The so-called \emph{Gaussian binomial coefficient} $\gaussm{n}{k}{q}$, where 
$\gaussm{n}{k}{q}=\prod_{i=n-k+1}^{n} (1-q^i)\,/\,\prod_{i=1}^{k} (1-q^i)$ for $0\le k\le n$ 
and $\gaussm{n}{k}{q}=0$ otherwise, gives the respective cardinality $\left|G_q(n,k)\right|$.  
A \emph{partial $k$-spread} in $\F{n}{q}$ is a collection of $k$-dimensional 
subspaces with trivial intersection such that each \emph{point}\footnote{In the projective space the dimensions 
are commonly one less compared to the consideration of subspaces in $\F{n}{q}$.}, 
i.e., each element of $G_q(n,1)$, is covered at most once. A point that is not covered by any of the 
$k$-dimensional subspaces of the partial $k$-spread is called a \emph{hole}. We call the number of $k$-dimensional subspaces 
of a given partial $k$-spread its size and we call it \emph{maximum} if it has the largest possible size.
Bounds for the sizes of maximum partial $k$-spreads were heavily studied in the past. Here we are able 
to determine the exact value for an infinite series of cases of parameters $n$ and $k$.

Besides the geometric interest in maximum partial $k$-spreads, they also can be seen as a special case of
$q$ subspace codes in (network) coding theory. Here the codewords are elements of $\PG{n-1}{q}$. Two widely 
used distance measures for subspace codes (motivated by an information-theoretic analysis of the 
Koetter--Kschischang--Silva model, see e.g.\ \cite{silva2008rank}) are the so-called
\emph{subspace distance} $d_S(U,V):=\dim(U+V)-\dim(U\cap V)=2\cdot\dim(U+V)-\dim(U)-\dim(V)$ and the so-called 
\emph{injection distance} $d_I(U,V):=\max\left\{\dim(U),\dim(V)\right\}-\dim(U\cap V)$. For $D\subseteq\{0,\dots,n\}$ 
we denote by $A_q(n,d;D)$ the maximum cardinality of a subspace code over $\F{n}{q}$ with minimum subspace distance 
at least $d$, where we additionally assume that the dimensions of the codewords are contained in $D$. The most unrestricted 
case is given by $D=\{0,\dots,n\}$. The other extreme, $D=\{k\}$ is called \emph{constant dimension} case and the corresponding 
codes are called \emph{constant dimension codes}. As an abbreviation we use the notation $A_q(n,d;k):=A_q(n,d;\{k\})$. Note 
that $d_S(U,V)=2\cdot d_I(U,V)\in 2\cdot\mathbb{N}$ in the constant dimension case. Bounds on $A_q(n,d;D)$ have been intensively 
studied in the last years, see e.g.\ \cite{etzion2013problems}. With this notation, the size of a maximum partial $k$-spread 
in $\F{n}{q}$ is given by $A_q(n,2k;k)$.

The remaining part of the paper is structured as follows. We will briefly review some known results on $A_q(n,2k;k)$ 
and discuss their relation with our main result in Section~\ref{sec_known_results}. In Section~\ref{sec_preparation} we 
will provide the technical tools that are then used to prove the main result in Section~\ref{sec_main_result}. We close 
with a conclusion listing some further implications and future lines of research in Section~\ref{sec_conclusion}.     

\section{Known bounds for partial spreads}
\label{sec_known_results}

\noindent
Counting the points in $\F{n}{q}$ and $\F{k}{q}$ gives the obvious upper bound $A_q(n,2k;k)\le \frac{\gaussm{n}{1}{q}}
{\gaussm{k}{1}{q}}=\frac{q^n-1}{q^k-1}$. If equality is attained, then one speaks of a \emph{$k$-spread}.

\begin{Theorem}
  \label{thm_spread}(\cite{andre1954nicht}; see also \cite[p.~29]{dembowski2012finite}, Result 2.1 in \cite{beutelspacher1975partial})
  $\mathbb{F}_q^n$ contains a $k$-spread if and only if $k$ divides $n$, where we assume $1\le k\le n$ and $k,n\in\mathbb{N}$.
\end{Theorem}

\noindent
If $k$ does not divide $n$, then we can improve the previous upper bound by rounding down to 
$A_q(n,2k;k)\le\left\lfloor\frac{q^n-1}{q^k-1}\right\rfloor$. Here a specific parameterization is useful: 
If one writes the size of a partial $k$-spread in $\F{n}{q}$, where $n=k(t+1)+r$, $1\le r\le k-1$, 
as $A_q(n,2k;k)=q^r\cdot\frac{q^{k(t+1)}-1}{q^k-1}-s$, then $s\ge q-1$ and $s>\frac{q^r-1}{2}-\frac{q^{2r-k}}{5}$ is known, see e.g. 
\cite{eisfeldt}. Furthermore, there exists an example with $s=q^r-1$ in each case, see e.g.\ Observation~\ref{obs_multi_component}, 
leading to the conjecture that the sharp bound is $s\ge q^r-1$. Assuming $q=2$ and $k\ge 4$, our main result in 
Theorem~\ref{thm_spread_exact_value_3} verifies this conjecture for $r=2$, i.e., $s\ge 3$.
Note that $n\equiv {r\pmod k}$, so that the residue class $r$ seems to play a major role. Besides the case of $r=0$, see 
Theorem~\ref{thm_spread}, the next case $r=1$ is solved in full generality:
\begin{Theorem}
  \label{thm_almost_spread}(\cite{beutelspacher1975partial}; see also \cite{hong1972general} for the special case $q=2$)
  %% essentially Theorem 4.1, see also Storme, Eisfeld lecture notes
  For integers $1\le k\le n$ with $n\equiv 1\pmod k$
  we have $A_q(n,2k;k)=\frac{q^n-q}{q^k-1}-q+1=q^1\cdot \frac{q^{n-1}-1}{q^k-1}-q+1=\frac{q^n-q^{k+1}+q^k-1}{q^k-1}$.
\end{Theorem}

The so far best upper bound on $A_q(n,2k;k)$, i.e., the best known lower bound on $s$ is based on:
\begin{Theorem} (Corollary 8 in \cite{nets_and_spreads})
  \label{thm_partial_spread_4}
  If $n=k(t+1)+r$ with $0<r<k$, then 
  $$
    A_q(n,2k;k)\le \sum_{i=0}^{t} q^{ik+r} -\left\lfloor\theta\right\rfloor-1
    =q^r\cdot \frac{q^{k(t+1)}-1}{q^k-1}-\left\lfloor\theta\right\rfloor-1,
  $$
  where $2\theta=\sqrt{1+4q^k(q^k-q^r)}-(2q^k-2q^r+1)$.
\end{Theorem}
We remark that this theorem is also restated as Theorem~13 in \cite{etzion2013problems} and as Theorem~44 in \cite{etzionsurvey} 
with the small typo of not rounding down $\theta$ ($\Omega$ in their notation). And indeed, the resulting lower bound 
$s\ge \left\lfloor\theta(q,k,r)\right\rfloor+1$ is independent of $n$.
%% 
%% Maple code:
%% ===========
%% restart;
%% n := 14; k := 5; q := 2; 
%% c := `mod`(n, k); l := (n-c)/k; 
%% Omega := floor((1/2)*(sqrt(1+4*q^k*(q^k-q^c))-2*q^k+2*q^c-1));
%% sum(q^(i*k+c), i = 0 .. l-1)-Omega-1;
Specializing to the binary case, i.e., $q=2$, we can use the previous results to state exact formulas for
$A_2(n,2k;k)$ for small values of $k\ge 2$.\footnote{Obviously, we have $A_q(n,2;1)=\gaussm{n}{1}{q}$.} 

From Theorem~\ref{thm_spread} and Theorem~\ref{thm_almost_spread} we conclude:
\begin{Corollary}
  For each integer $m\ge 2$ we have
  %% A_2(3,4;2)=1 ist auch durch die Formeln abgedeckt.
  \begin{itemize}
    \item[(a)] $A_2(2m,4;2)=\frac{2^{2m}-1}{3}$;
    \item[(b)] $A_2(2m+1,4;2)=\frac{2^{2m+1}-5}{3}$.
  \end{itemize}
\end{Corollary}

Using the results of Theorem~\ref{thm_spread}, Theorem~\ref{thm_almost_spread}, and Theorem~\ref{thm_partial_spread_4} 
the case $k=3$ was completely settled in \cite{spreadsk3}:
\begin{Theorem}
  \label{thm_spread_k_3}
  For each integer $m\ge 2$ we have
  \begin{itemize}
    \item[(a)] $A_2(3m,6;3)=\frac{2^{3m}-1}{7}$;
    \item[(b)] $A_2(3m+1,6;3)=\frac{2^{3m+1}-9}{7}$;
    \item[(c)] $A_2(3m+2,6;3)=\frac{2^{3m+2}-18}{7}$.
  \end{itemize}
\end{Theorem}

In our Theorem~\ref{thm_spread_exact_value_3} we completely settle the case $n\equiv 2\pmod k$ for $q=2$, $k\ge 4$, and
$n\ge 2k+2$.\footnote{As $A_q(k+2,2k;k)=1$ for $k\ge 2$, the assumption $n\ge 2k+2$ is no restriction. The case $k=3$ 
is covered by \cite{spreadsk3}, see Theorem~\ref{thm_spread_k_3}. For $k=1,2$ the remainder of $n$ is strictly smaller than $2$. 
So, in other words, the binary case $n\equiv 2\pmod k$ is completely resolved.} Using the results of Theorem~\ref{thm_spread}, 
Theorem~\ref{thm_almost_spread}, Theorem~\ref{thm_partial_spread_4}, Observation~\ref{obs_multi_component}, and 
Theorem~\ref{thm_spread_exact_value_3} we can state:  
\begin{Corollary}
  \label{cor_spread_k_4}
  For each integer $m\ge 2$ we have
  \begin{itemize}
    \item[(a)] $A_2(4m,8;4)=\frac{2^{4m}-1}{15}$;
    \item[(b)] $A_2(4m+1,8;4)=\frac{2^{4m+1}-17}{15}$;
    \item[(c)] $A_2(4m+2,8;4)=\frac{2^{4m+2}-49}{15}$;
    \item[(d)] $\frac{2^{4m+3}-113}{15}\le A_2(4m+3,8;4)\le\frac{2^{4m+3}-53}{15}$.
  \end{itemize}
\end{Corollary}

In \cite{etzion2013problems} Etzion listed 100 open problems on $q$-analogs in coding theory. Our main theorem resolves several of them: 
\begin{itemize}
  \item Research problem 45 asks for a characterization of parameter cases for which the construction in 
        Observation~\ref{obs_multi_component} matches the exact value of $A_q(n,2k;k)$. Assuming $q=2$ and $k\ge 4$, this is 
        the case for $n\equiv 2\pmod k$.
  \item Research problem 46 asks for improvements of the upper bound from Theorem~\ref{thm_partial_spread_4}, which are 
        achieved for the same parameters as specified above. The same is true for Research problem 47 asking for exact values.
  \item The special case of the determination of $A_2(n,8;4)$ in Research problem 49 is completely resolved for $n\equiv 2\pmod 4$, 
        see Corollary~\ref{cor_spread_k_4}.               
\end{itemize}

\section{Constructions and vector space partitions}
\label{sec_preparation}

\noindent
For matrices $A,B\in\mathbb{F}_q^{m\times n}$ the \emph{rank distance} is defined via $d_R(A,B):=\operatorname{rk}(A-B)$. It is indeed 
a metric, as observed in \cite{gabidulin1985theory}.

\begin{Theorem}(see \cite{gabidulin1985theory})
  \label{thm_MRD_size}
  Let $m,n\ge d$ be positive integers, $q$ a prime power, and $\mathcal{C}\subseteq \mathbb{F}_q^{m\times n}$ be a rank-metric 
  code with minimum rank distance $d$. Then, $|\mathcal{C}|\le q^{\max(n,m)\cdot (\min(n,m)-d+1)}$. 
  Codes attaining this upper bound are called maximum rank distance (MRD) codes. They exist for all (suitable) choices of parameters. 
\end{Theorem}  
If $m<d$ or $n<d$, then only $|\mathcal{C}|=1$ is possible, which may be summarized to the single upper bound 
$|\mathcal{C}|\le \left\lceil q^{\max(n,m)\cdot (\min(n,m)-d+1)}\right\rceil$. 
Using an $m\times m$ identity matrix as a prefix one obtains the so-called \emph{lifted MRD codes}.

\begin{Theorem}(see \cite{silva2008rank})
  For positive integers $k,d,n$ with $k\le n$, $d\le 2\min(k,n-k)$, and $d\equiv 0\pmod{2}$, the size of a lifted MRD code in $G_q(n,k)$ with 
  subspace distance $d$ is given by $$M(q,k,n,d):=q^{\max(k,n-k)\cdot(\min(k,n-k)-d/2+1)}.$$ If $d>2\min(k,n-k)$, then we have $M(q,k,n,d)=1$. 
\end{Theorem}

In \cite{etzion2009error} a generalization, the so-called multi-level construction,  was presented. To this end, let $1\le k\le n$ be integers and $v\in \mathbb{F}_2^n$ a
binary vector of weight $k$. By $\operatorname{EF}_q(v)$ we denote the set of all $k\times n$ matrices over $\mathbb{F}_2$ 
that are in row-reduced echelon form, i.e., the Gaussian algorithm had been applied, and the pivot columns coincide with the 
positions where $v$ has a $1$-entry.

\begin{Theorem} (see \cite{etzion2009error})
  \label{thm_echelon_ferrers}
  For integers $k,n,d$ with $1\le k\le n$ and $1\le d\le \min(k,n-k)$, let $\mathcal{B}$ be a binary constant weight code of length $n$, 
  weight $k$, and minimum Hamming distance $2d$. For each $b\in \mathcal{B}$ let $\mathcal{C}_b$ be a code in $\operatorname{EF}_q(b)$ 
  with minimum rank distance at least $d$. Then, $\cup_{b\in\mathcal{B}} \,\mathcal{C}_b$ is a constant dimension code
  of dimension $k$ having a subspace distance of at least $2d$.
\end{Theorem}  

The authors of \cite{etzion2009error} also came up with a conjecture for the size of an MRD code in $\operatorname{EF}_q(v)$, which is still unrebutted. 
Taking binary vectors with $k$ consecutive ones we are in the classical MRD case. So, taking binary vectors $v_i$, where the ones 
are located in positions $(i-1)k+1$ to $ik$, for all $1\le i\le \left\lfloor n/k\right\rfloor$, clearly gives 
a binary constant weight code of length $n$, weight $k$, and minimum Hamming distance $2k$.  

\begin{Observation}
  \label{obs_multi_component} 
  For positive integers $k$, $n$ with $n>2k$ and $n\not\equiv 0\pmod{k}$, there exists a constant dimension code in $G_q(n,k)$ 
  with subspace distance $2k$ having cardinality\footnote{Using our general notation, we
    may rewrite the stated formula with $n=k(t+1)+r$ and $n\,\operatorname{mod}\,k=r$.}
  $$
    1+\sum_{i=1}^{\left\lfloor n/k\right\rfloor-1} q^{n-ik}
    =1+q^{k+(n\,\operatorname{mod}\,k)}\cdot\frac{q^{n-k-(n\,\operatorname{mod}\,k)}-1}{q^k-1}
    =\frac{q^n-q^{k+(n\,\operatorname{mod}\,k)}+q^k-1}{q^k-1}.
  $$ 
\end{Observation}
We remark that a more general construction, among similar lines and including explicit formulas for the respective cardinalities, has 
been presented in \cite{skachek2010recursive}.

\bigskip
\bigskip

%%\section{Vector space partitions}

A \emph{vector space partition} $\mathcal{P}$ of $\F{n}{q}$ is a collection of subspaces with the property that every nonzero vector  
of $\F{n}{q}$ is contained in a unique member of $\mathcal{P}$. If for $d\in\{1,2,\dots,k\}$ the  vector space 
partition $\mathcal{P}$  contains $m_d$ subspaces of dimension $d$ and $m_k>0$, then $(m_k,m_{k-1},\dots,m_1)$ is called the 
\emph{type} of $\mathcal{P}$. We will also use the notation $k^{m_k}\dots 1^{m_1}$, where we may leave out cases with $m_d=0$. 
The \emph{tail} of $\mathcal{P}$ is the set of subspaces, in $\mathcal{P}$, having the smallest dimension. If the dimension of the 
corresponding subspaces is given by $d$, then the \emph{length} of the tail is the number $m_d$, i.e., the cardinality of the tail.

\begin{Theorem} (Theorem 1 in \cite{heden2009length})
  \label{thm_length_of_tail}
  Let $\mathcal{P}$ be a vector space partition of $\F{n}{q}$, let $n_1$ denote the length of the tail of $\mathcal{P}$, let 
  $d_1$ denote the dimension of the vector spaces in the tail of $\mathcal{P}$, and let $d_2$ denote the dimension of the 
  vector spaces of the second lowest dimension.
  \begin{enumerate}
    \item[(i)]   if $q^{d_2-d_1}$ does not divide $n_1$ and if $d_2<2d_1$, then $n_1\ge q^{d_1}+1$;
    \item[(ii)]  if $q^{d_2-d_1}$ does not divide $n_1$ and if $d_2\ge 2d_1$, then either $d_1$ divides $d_2$ and 
                 $n_1=\left(q^{d_2}-1\right)/\left(q^{d_1}-1\right)$ or $n_1>2q^{d_2-d_1}$;
    \item[(iii)] if $q^{d_2-d_1}$ divides $n_1$ and $d_2<2d_1$, then $n_1\ge q^{d_2}-q^{d_1}+q^{d_2-d_1}$;
    \item[(iv)]  if $q^{d_2-d_1}$ divides $n_1$ and $d_2\ge 2d_1$, then $n_1\ge q^{d_2}$.
  \end{enumerate}   
\end{Theorem}   
So, in any (nontrivial) case\footnote{We have to exclude the trivial subspace partition $\mathcal{P}=\left\{\F{n}{q}\right\}$, where 
$d_1=n$ and $d_2$ does not exist.}, we have $n_1\ge q+1\ge 3$, which will be sufficient in many situations.

\section{Main theorem}
\label{sec_main_result}

For a vector space partition $\mathcal{P}$ of $\F{n}{q}$ and a hyperplane $H$, let 
$\mathcal{P}_H:=\{U\cap H\,:\,U\in\mathcal{P}\}$ be the vector space partition of $\F{n-1}{q}$, i.e., $\mathcal{P}_H$ is 
obtained from $\mathcal{P}$ by the intersection with hyperplane $H$. 

\noindent
\begin{Lemma}
  \label{lemma_forbidden_type_6}
  For two integers $t\ge 1$ and $k\ge 4$,  no vector space partition of type $k^{n_k} (k-1)^{n_{k-1}} 1^{1+2^{k-1}}$ exists in 
  $\F{k(t+1)+1}{2}$, where $n_k=\frac{2^{kt+2}+2^{k}-5}{2^k-1}$ and $n_{k-1}=2^{kt+2}-3$.\footnote{Theorem~\ref{thm_length_of_tail}.(ii,iv) 
  yields $n_1= 2^{k-1}-1$ or $n_1>2^{k-1}$, if we set $d_2=k-1$ and $d_1=1$. The improvement of Theorem~\ref{thm_length_of_tail}, i.e.\ see 
   \cite[Theorem~2]{heden2013supertail}, is not sufficient to exclude the case of Lemma~\ref{lemma_forbidden_type_6}.} 
\end{Lemma}
\begin{proof}
  Assume the existence of a vector space partition $\mathcal{P}$ of the specified type. Let $H$ be an arbitrary hyperplane.  
  Since the $m=\frac{2^{k(t+1)+2}-2^{k+1}-2}{2^k-1}$ 
  non-holes of $\mathcal{P}_H$  have dimensions in $\{k,k-1,k-2\}$ and the total number of points is given by $\gaussm{k(t+1)}{1}{2}=2^{k(t+1)}-1$, 
  the number of holes $L_H$ has to satisfy $L_H\equiv 1\pmod {2^{k-2}}$. Using $L_H\le 1+2^{k-1}$, we conclude 
  $L_H\in \{1,1+2^{k-2},1+2^{k-1}\}$. Due to the tail condition in Theorem~\ref{thm_length_of_tail}, the case $L_H=1$ is impossible. 
  Now let $x$ be the number of hyperplanes with $L_H=1+2^{k-1}$ holes and $\gaussm{k(t+1)+1}{k(t+1)}{2}-x=2^{k(t+1)+1}-1-x$ the number 
  of hyperplanes with $L_H=1+2^{k-2}$ holes. Since each hole is contained in $\gaussm{k(t+1)}{k(t+1)-1}{2}=2^{k(t+1)}-1$ hyperplanes, we have
  \begin{eqnarray*}
    &&\frac{\left(1+2^{k-1}\right)x+\left(1+2^{k-2}\right)\cdot (2^{k(t+1)+1}-1-x)}{2^{k(t+1)}-1}\\
    &=&\frac{\left(1+2^{k-2}\right)\cdot 2^{k(t+1)+1}-\left(1+2^{k-2}\right)+2^{k-2}\cdot x}{2^{k(t+1)}-1}\\
    &\ge& \frac{\left(1+2^{k-2}\right)\cdot 2^{k(t+1)+1}-\left(1+2^{k-2}\right)}{2^{k(t+1)}-1}\\
    &>& 2\cdot \left(1+2^{k-2}\right)=2^{k-1}+2>1+2^{k-1}   
  \end{eqnarray*}
  holes in total, a contradiction.
\end{proof}

\begin{Lemma}
  \label{lemma_special_theta_bound}
  Using the notation from Theorem~\ref{thm_partial_spread_4}, we have $\left\lfloor\theta\right\rfloor
  =\left\lfloor\frac{q^r-2}{2}\right\rfloor$ for $r\ge 1$ and $k\ge 2r$.\footnote{The result is also valid 
  for $k=2r-1$, $r\ge 2$, and $q\in\{2,3\}$.} 
\end{Lemma}
\begin{Proof}
  We have
  $$
    2\theta=\sqrt{1+4q^k(q^k-q^r)}-(2q^k-2q^r+1)
    =\sqrt{\left(2q^k-q^r\right)^2-q^{2r}+1}-(2q^k-2q^r+1)<q^r-1.
  $$
  Since $1+4q^k(q^k-q^r)= 1+4q^{2k}-4q^{k+r}>
  \left(2q^k-(q^r+1)\right)^2 = 4q^{2k} -4q^{k+r} -4q^k+ q^{2r}+2q^r+1$ for $k\ge 2r$ and $q\ge 2$, we have
  $2\theta >q^r-2$. Thus, we have $\left\lfloor\theta\right\rfloor=(q^r-2)/2$ for $q$ even and 
  $\left\lfloor\theta\right\rfloor=(q^r-3)/2$ for $q$ odd. 
\end{Proof}

We remark that the formula for $\left\lfloor\theta\right\rfloor$ in Lemma~\ref{lemma_special_theta_bound} does not 
depend on $k$ (supposing that $k$ is sufficiently large).

\begin{Theorem}
  \label{thm_spread_exact_value_3}
  For integers $t\ge 1$ and $k\ge 4$, we have $A_2(k(t+1)+2,2k;k)=\frac{2^{k(t+1)+2}-3\cdot 2^{k}-1}{2^k-1}$.
\end{Theorem}
\begin{proof}
  Applying Lemma~\ref{lemma_special_theta_bound} and Theorem~\ref{thm_partial_spread_4} yields 
  $$A_2(k(t+1)+2,2k;k)\le \frac{2^{k(t+1)+2}-2^{k+1}-2}{2^k-1}.$$
  Assuming that the upper bound $m:=\frac{2^{k(t+1)+2}-2^{k+1}-2}{2^k-1}$ is attained,  %%by a code $\mathcal{C}$, 
  we obtain 
  a vector space partition $\mathcal{P}$ of type $k^{m} 1^{2^{k+1}+1}$, i.e., the $m$ $k$-dimensional codewords leave over 
  $\gaussm{k(t+1)+2}{1}{2}-m\cdot \gaussm{k}{1}{2}=2^{k(t+1)+2}-1-\frac{2^{k(t+1)+2}-2^{k+1}-2}{2^k-1}
  \cdot \left(2^{k}-1\right)=2^{k+1}+1$ holes. Now we consider the intersection of $\mathcal{P}$ with a hyperplane $H$. 
  Since the codewords end up as $k$- or $(k-1)$-dimensional subspaces summing up to $m$ , the number of holes is at most 
  $2^{k+1}+1$, and the total number of points is given by $\gaussm{k(t+1)+1}{1}{2}=2^{k(t+1)+1}-1$, we obtain the following 
  list of possible types of $\mathcal{P}_H$:
  \begin{enumerate}
    \item $k^{n_k+1} (k-1)^{n_{k-1}-1} 1^1$
    \item $k^{n_k} (k-1)^{n_{k-1}} 1^{1+2^{k-1}}$
    \item $k^{n_k-1} (k-1)^{n_{k-1}+1} 1^{1+2^{k}}$
    \item $k^{n_k-2} (k-1)^{n_{k-1}+2} 1^{1+3\cdot 2^{k-1}}$
    \item $k^{n_k-3} (k-1)^{n_{k-1}+3} 1^{1+2^{k+1}}$,
  \end{enumerate}
  where $n_k=\frac{2^{kt+2}+2^{k}-5}{2^k-1}$ and $n_{k-1}=2^{kt+2}-3$.    

  Due to Theorem~\ref{thm_length_of_tail}, case~(1) is impossible. The case~(2) is ruled out by Lemma~\ref{lemma_forbidden_type_6}. 
  Thus, each of the $\gaussm{k(t+1)+2}{k(t+1)+1}{2}=2^{k(t+1)+2}-1$ hyperplanes contains at most $n_k-1$ subspaces of dimension $k$. 
  Since each $k$-dimensional subspace is contained in $\gaussm{kt+2}{kt+1}{2}=2^{kt+2}-1$ hyperplanes, the total number of $k$-dimensional 
  subspaces in $\mathcal{P}$ can be at most
  \begin{eqnarray*}
    \frac{\left(2^{k(t+1)+2}-1\right)\cdot \left(n_k-1\right)}{2^{kt+2}-1}
    &=&\frac{2^{k(t+1)+2}-1}{2^k-1}-3\cdot \frac{2^{k(t+1)+2}-1}{\left(2^k-1\right)\cdot\left(2^{kt+2}-1\right)}\\  
    &\overset{k>0}{<}&\frac{2^{k(t+1)+2}-3\cdot 2^k-1}{2^k-1},
  \end{eqnarray*}
  a contradiction. Thus we have $A_2(k(t+1)+2,2k;k)\le \frac{2^{k(t+1)+2}-3\cdot 2^{k}-1}{2^k-1}$. A 
  construction for $A_2(k(t+1)+2,2k;k)\ge \frac{2^{k(t+1)+2}-3\cdot 2^{k}-1}{2^k-1}$ is given by
  Observation~\ref{obs_multi_component}.
\end{proof}

\begin{Corollary}
  \label{cor_spread_exact_value_1}
  For each integer $k\ge 4$ we have $A_2(2k+2,2k;k)=2^{k+2}+1$.
\end{Corollary}

We remark that Corollary~\ref{cor_spread_exact_value_1} would be wrong for $k=3$, since $A_2(8,6;3)=34>33$, see \cite{spreadsk3}.
And indeed, each extremal code has to contain a hyperplane which is a subspace partition of type $3^5 2^{29} 1^{5}$. Next we try to get 
a bit more information about these extremal codes. To this end, let $a_i$ denote the number of hyperplanes containing exactly $2\le i\le 5$ 
three-dimensional codewords and $17\ge 25-4i>1$ holes. The \emph{standard equations} for our parameters are given by
\begin{eqnarray*}
  a_2+a_3+a_4+a_5     &=& \gaussm{8}{7}{2}=255\\
  2a_2+3a_3+4a_4+5a_5 &=& \gaussm{5}{1}{2}\cdot A_2(8,6;3)=1054\\
  a_2+3a_3+6a_4+10a_5 &=& {{A_2(8,6;3)} \choose 2}=1683.
\end{eqnarray*}
Solving the equation system in terms of $a_5$ yields $a_2=51-a_5$, $a_3=3a_5-136$, and $a_4=340-3a_5$.
Since the $a_i$ have to be non-negative, we obtain $46\le a_5\le 51$ and $0\le a_2\le 5$. Now let $L$ 
be the subspace generated by the $17$ holes. Since $17>15$ we have $\dim(L)\in \{5,6,7,8\}$. A hyperplane 
containing $2$ codewords contains all $17$ holes so that the set of hyperplanes of this type corresponds to 
the set of hyperplanes containing $L$ as a subspace, i.e., $\dim(L)=8-i$ is equivalent to $a_2=2^i-1$ for 
$0\le i\le 3$. Thus, the list of theoretically possible spectra is given by
$(0,17,187,51)$, $(1,14,190,50)$, and $(3,8,196,48)$, i.e., there are at least $48$ hyperplanes of type 
$3^52^{29}1^5$.

We remark that Lemma~\ref{lemma_forbidden_type_6} can be generalized to arbitrary odd\footnote{For even $q>2$ the tail condition of
Theorem~\ref{thm_length_of_tail} cannot be applied directly in the proof of Lemma~\ref{lemma_forbidden_type_7}.} prime powers $q$ 
along the same lines:
\begin{Lemma}
  \label{lemma_forbidden_type_7}
  For integers $t\ge 1$, $k\ge 4$, and odd $q$, no vector space partition of type $k^{p-1} (k-1)^{m-p+1} 1^{\frac{q+1}{2}+q^{k-1}}$ exists in 
  $\F{k(t+1)+1}{q}$, where $p=\frac{q^{kt+2}-q^2}{q^k-1}+\frac{q+1}{2}$ and $m=\frac{q^{k(t+1)+2}-q^2}{q^k-1}-\frac{q^2-1}{2}$.
\end{Lemma}
\begin{proof}
  Assume the existence of a vector space partition $\mathcal{P}$ of the specified type. Now we consider the intersection 
  with an arbitrary hyperplane $H$. Since the non-holes of $\mathcal{P}$ end up as $m$ subspaces, with dimensions in 
  $\{k,k-1,k-2\}$, in $\mathcal{P}_H$ and the total number of points is given by $\gaussm{k(t+1)}{1}{q}$, the number of holes 
  $L_H$ in $\mathcal{P}_H$ has to satisfy $L_H\equiv \frac{q+1}{2}\pmod {q^{k-2}}$. Using $L_H\le \frac{q+1}{2}+q^{k-1}$, we conclude
  $L_H\in \left\{\frac{q+1}{2}+iq^{k-2}\,:\, 0\le i\le q\right\}$. 
  %%$L\in \{\frac{q+1}{2},\frac{q+1}{2}+q^{k-2},\frac{q+1}{2}+q^{k-1}\}$. 
  Due to the tail condition in Theorem~\ref{thm_length_of_tail}, 
  the case $L_H=\frac{q+1}{2}$ is impossible. Thus, each hyperplane contains at least $\frac{q+1}{2}+q^{k-2}$ holes.  
  %% Now let $x$ be the number of hyperplanes with $\frac{q+1}{2}+q^{k-1}$ holes and 
  %% $\gaussm{k(t+1)+1}{k(t+1)}{q}-x$ the number of hyperplanes with $\frac{q+1}{2}+q^{k-2}$ holes. 
  Since each hole is contained in $\gaussm{k(t+1)}{k(t+1)-1}{q}$ hyperplanes, we have at least
  $$
    \left(\frac{q+1}{2}+q^{k-2}\right) \cdot \frac{q^{k(t+1)+1}-1}{q^{k(t+1)}-1}
    \,\ge\, \left(\frac{q+1}{2}+q^{k-2}\right)\cdot q\,>\, \frac{q+1}{2}+q^{k-1}
  $$
  holes in total, a contradiction.
\end{proof}

It turns out that repeating the proof of Theorem~\ref{thm_spread_exact_value_3} for odd $q$ just works for $q=3$ and additionally 
the lower bound by the construction of Observation~\ref{obs_multi_component} does not match the improved upper bound. At the 
very least an improvement of the upper bound of Theorem~\ref{thm_partial_spread_4} by one is possible: 
\begin{Lemma}
  \label{lemma_spread_upper_bound_3_q}
  For integers $t\ge 1$ and $k\ge 4$, we have $A_3(k(t+1)+2,2k;k) \le \frac{3^{k(t+1)+2}-3^2}{3^k-1}-\frac{3^2+1}{2}$.
\end{Lemma}
\begin{Proof}
  Applying Lemma~\ref{lemma_special_theta_bound} and Theorem~\ref{thm_partial_spread_4} for odd $q$ yields 
  $$
    A_q(k(t+1)+2,2k;k)\le \frac{q^{k(t+1)+2}-q^2}{q^k-1}-\frac{q^2-1}{2}=:m.
  $$
  Assuming that the upper bound is attained by a code $\mathcal{C}$, the $m$ $k$-dimensional codewords leave at least 
  $$
    \gaussm{k(t+1)+2}{1}{q}-m\cdot \gaussm{k}{1}{q}
    =\frac{q(q+1)}{2}\cdot q^{k-1} +\frac{q+1}{2}=:h
  $$ 
  holes. Now we consider the intersection of $\mathcal{C}$ with a hyperplane. 
  Since the codewords end up as $k$- or $(k-1)$-dimensional subspaces summing up to $m$ , the number of holes is at most 
  $h$, and the total number of points is given by $\gaussm{k(t+1)+1}{1}{q}=\frac{q^{k(t+1)+1}-1}{q-1}$, we obtain the types
  $$
    k^{p-i} (k-1)^{m-p+i} 1^{\frac{q+1}{2}+iq^{k-1}}
  $$
  for $0\le i\le \frac{q(q+1)}{2}$, where $p:=\frac{q^{kt+2}-q^2}{q^k-1}+\frac{q+1}{2}$.
  
  Due to Theorem~\ref{thm_length_of_tail}, case $i=0$ is impossible. The case $i=1$ is ruled out by Lemma~\ref{lemma_forbidden_type_7}. 
  Thus, each of the $\gaussm{k(t+1)+2}{k(t+1)+1}{q}$ hyperplanes contains at most $p-2$ subspaces of dimension $k$. 
  Since each $k$-dimensional subspace is contained in $\gaussm{kt+2}{kt+1}{q}$ hyperplanes, the total number of $k$-dimensional 
  subspaces in $\mathcal{C}$ can be at most
  \begin{eqnarray*}
     && \frac{(p-2)\cdot \gaussm{k(t+1)+2}{k(t+1)+1}{q}}{\gaussm{kt+2}{kt+1}{q}}
     = \frac{\left( \frac{q^{kt+2}-q^2}{q^k-1}+\frac{q-3}{2} \right)\cdot\left( q^k(q^{kt+2}-1)+q^k-1 \right)}
    {q^{kt+2}-1}\\
    &=& \frac{q^{k(t+1)+2}-q^2-q^{k+2}+q^2}{q^k-1}+\frac{q-3}{2}\cdot q^k +\frac{q^{kt+2}-q^2+\frac{q-3}{2}\cdot\left(q^k-1\right)}
    {q^{kt+2}-1}\\
    &\overset{q=3}{=}& \frac{q^{k(t+1)+2}-q^2}{q^k-1}-q^2+
    \frac{q^{kt+2}-q^2}{q^{kt+2}-1}\\
    &<& \frac{q^{k(t+1)+2}-q^2}{q^k-1}-q^2+1
    \overset{q>1}{<} \frac{q^{k(t+1)+2}-q^2}{q^k-1}-\frac{q^2-1}{2}=m       
  \end{eqnarray*}
  a contradiction. Thus we have $A_3(k(t+1)+2,2k;k)\le \frac{3^{k(t+1)+2}-3^2}{3^k-1}-\frac{3^2+1}{2}$.
  %=\frac{3^{k(t+1)+2}-5\cdot 3^{k}-4}{3^k-1}$. 
\end{Proof}

\section{Conclusion}
\label{sec_conclusion}

\noindent
For the size of a maximum partial $k$-spread in $\F{n}{q}$ the exact formula
$A_q(k(t+1)+r,2k;k)=q^r\cdot\frac{q^{k(t+1)}-1}{q^k-1}-q^r+1$ was conjectured for some time, where 
$n=k(t+1)+r$ and $1\le r\le k-1$. Codes with these parameters can easily be obtained via combining 
some MRD codes, see Observation~\ref{obs_multi_component}. However, the conjecture is false for $q=2$, $k=3$, $n\equiv 2\pmod 3$, 
and $n\ge 8$, as we know since \cite{spreadsk3}. In this paper we have shown that the conjecture is true for $q=2$, $k\ge 4$,  
$n\equiv 2\pmod k$, and $n\ge 2k+2$. With respect to upper bounds, Theorem~\ref{thm_partial_spread_4} is one of the 
most general and sweeping theoretical tools. For the spread case, i.e., $n\equiv 0\pmod k$, it was sufficient to consider 
the (empty) set of holes. The main idea of Beutelspacher for the case $n\equiv 1\pmod k$, may roughly be described 
as the consideration of holes in the projections of partial $k$-spreads in hyperplanes. In this sense, our work 
is just the continuation of projecting two times.\footnote{The specific use of Theorem~\ref{thm_length_of_tail} is 
just a shortcut, resting on the same rough idea. However, it points to an area where even more theoretic results are 
available, that possibly can be used in more involved cases.} If $k\ge 4$ the projected codewords can be distinguished from the 
holes by the attained dimensions. So, we naturally ask whether our result can be generalized to arbitrary $q$. In 
Lemma~\ref{lemma_spread_upper_bound_3_q} we were able to reduce the previously best known upper bound by $1$ for the 
special field size $q=3$. Looking closer at our arguments shows that for further progress additional ideas are needed. 

In general, 
one may project $k-2$ times without being confronted with an interference between the projected codewords and the set of holes 
contained in the $(n-k+2)$-dimensional subspaces. Can this rough idea be used to obtain improved upper bounds for $r\ge 3$ and 
$k\ge r+2$?\footnote{In this context, we would like to mention the very recent preprints 
\cite{kurz_2016_upper_bounds_partial_spreads,nastase2016maximum} including the bound $A_2(11,8;4)\le 132$.} 

Our main result suggests that the code attaining $A_2(8,6;3)=34$ is somehow \textit{specific}. As mentioned before, 
it cannot be obtained by the construction from Observation~\ref{obs_multi_component}. Even more, it cannot be obtained 
by the more general, so-called, Echelon-Ferrers (or multi-level) construction from \cite{etzion2009error}. So, a better understanding 
of the corresponding codes might be the key for possibly better constructions beating the currently best known 
lower bounds for e.g.\  $A_2(11,8;4)$ or $A_2(14,10;5)$. 

We would like to mention a new on-line table 
for upper and lower bounds for subspace codes at 
\begin{center}
\url{http://subspacecodes.uni-bayreuth.de}, 
\end{center}
see also \cite{TableSubspacecodes} for a brief manual and description of the methods implemented so far.
Actually, our research was 
initiated by looking for the smallest set of parameters, in the binary partial spread case, where the currently known lower and 
the upper bounds differ by exactly $1$: $65\le A_2(10,8;4)\le 66$. The other cases with a difference of one are exactly those 
that we finally covered by Theorem~\ref{thm_spread_exact_value_3}. Now, the \textit{smallest} unknown maximal cardinality of 
a partial $k$-spread over $\F{n}{2}$ is given by $129 \le A_2(11,8;4)\le 133$ and also the other cases, where the upper and the 
lower bound are exactly $4$ apart, show an obvious pattern. At least for us, the mentioned database was very valuable. 
As it commonly happens that formerly known results were rediscovered by different authors, we would appreciate any 
comments on existing results, that are not yet included in the database, very much.  

Partial $k$-spreads have applications in the construction of orthogonal arrays and $(\overline{s},\overline{r},\mu)$-nets\footnote{
Using the notation from this paper, we have $\overline{s}=q^k$, $\overline{r}=A_q(n,2k;k)$, and $\mu=q^{n-2k}$.}, see \cite{nets_and_spreads}. 
Thus, Theorem~\ref{thm_spread_exact_value_3} also implies restrictions for these objects. The derivation of the explicit corollaries 
goes along the same lines as presented in \cite{spreadsk3}.

\section*{Acknowledgements}

\noindent
The author thanks the referees for carefully reading a preliminary version of this article and giving very useful comments on its 
presentation.

%\bibliographystyle{amsplain}
%\bibliography{work_spread}

\begin{thebibliography}{10}

\bibitem{andre1954nicht}
J.~Andr{\'e}, \emph{{\"U}ber nicht-desarguessche {E}benen mit transitiver
  {T}ranslationsgruppe}, Mathematische Zeitschrift \textbf{60} (1954), no.~1,
  156--186.

\bibitem{beutelspacher1975partial}
A.~Beutelspacher, \emph{Partial spreads in finite projective spaces and partial
  designs}, Mathematische Zeitschrift \textbf{145} (1975), no.~3, 211--229.

\bibitem{dembowski2012finite}
P.~Dembowski, \emph{Finite {G}eometries: {R}eprint of the 1968 edition},
  Springer Science \& Business Media, 2012.

\bibitem{nets_and_spreads}
D.A. Drake and J.W. Freeman, \emph{Partial $t$-spreads and group constructible
  $(s,r,\mu)$-nets}, Journal of Geometry \textbf{13} (1979), no.~2, 210--216.

\bibitem{eisfeldt}
J.~Eisfeld and L.~Storme, \emph{$t$-spreads and minimal $t$-covers in finite
  projective spaces}, Lecture notes, Universiteit Gent, 29 pages (2000).

\bibitem{spreadsk3}
S.~El-Zanati, H.~Jordon, G.~Seelinger, P.~Sissokho, and L.~Spence, \emph{The
  maximum size of a partial $3$-spread in a finite vector space over
  ${G}{F}(2)$}, Designs, Codes and Cryptography \textbf{54} (2010), no.~2,
  101--107.

\bibitem{etzion2013problems}
T.~Etzion, \emph{Problems on $q$-analogs in coding theory}, arXiv preprint:
  1305.6126, 37 pages (2013).

\bibitem{etzion2009error}
T.~Etzion and N.~Silberstein, \emph{Error-correcting codes in projective spaces
  via rank-metric codes and {F}errers diagrams}, IEEE Transactions on
  Information Theory \textbf{55} (2009), no.~7, 2909--2919.

\bibitem{etzionsurvey}
T.~Etzion and L.~Storme, \emph{Galois geometries and coding theory}, Designs,
  Codes and Cryptography \textbf{78} (2016), no.~1, 311--350.

\bibitem{gabidulin1985theory}
E.M. Gabidulin, \emph{Theory of codes with maximum rank distance}, Problemy
  Peredachi Informatsii \textbf{21} (1985), no.~1, 3--16.

\bibitem{heden2009length}
O.~Heden, \emph{On the length of the tail of a vector space partition},
  Discrete Mathematics \textbf{309} (2009), no.~21, 6169--6180.

\bibitem{heden2013supertail}
O.~Heden, J.~Lehmann, E.~N{\u{a}}stase, and P.~Sissokho, \emph{The supertail of
  a subspace partition}, Designs, Codes and Cryptography \textbf{69} (2013),
  no.~3, 305--316.

\bibitem{TableSubspacecodes}
D.~Heinlein, M.~Kiermaier, S.Kurz, and A.~Wassermann, \emph{Tables of
  subspace codes}, University of Bayreuth, 2015, available at
  http://subspacecodes.uni-bayreuth.de.

\bibitem{hong1972general}
S.J. Hong and A.M. Patel, \emph{A general class of maximal codes for computer
  applications}, IEEE Transactions on Computers \textbf{100} (1972), no.~12,
  1322--1331.

\bibitem{kurz_2016_upper_bounds_partial_spreads}
S.~Kurz, \emph{Upper bounds for partial spreads}, arXiv preprint 1606.08581 (2016).

\bibitem{nastase2016maximum}
E.~N{\u{a}}stase and P.~Sissokho, \emph{The maximum size of a partial spread in a
  finite projective space}, arXiv preprint 1605.04824 (2016).

\bibitem{silva2008rank}
D.~Silva, F.R. Kschischang, and R.~Koetter, \emph{A rank-metric approach to
  error control in random network coding}, IEEE Transactions on Information
  Theory \textbf{54} (2008), no.~9, 3951--3967.

\bibitem{skachek2010recursive}
V.~Skachek, \emph{Recursive code construction for random networks}, IEEE
  transactions on Information Theory \textbf{56} (2010), no.~3, 1378--1382.

\end{thebibliography}

\end{document}